\documentclass[12pt]{amsart}
\usepackage{graphicx}
\usepackage[active]{srcltx}
\usepackage{amsthm,mathrsfs}
\usepackage{a4wide}
\usepackage{amsfonts}
\usepackage{amsmath}
\usepackage{amssymb}
\usepackage{dsfont}
\newtheorem{lemma}{Lemma}
\newtheorem{theorem}{Theorem}
\newtheorem{corollary}{Corollary}

\usepackage{float}
\usepackage{url}
\usepackage{enumerate}
\usepackage{enumerate}
\usepackage{enumitem}

\DeclareMathOperator {\vart} {Var_{total}}
\DeclareMathOperator {\varhk} {Var_{HK}}
\DeclareMathOperator {\varhkn} {Var_{HK\mathbf{0}}}
\newcommand {\ve} {\varepsilon}

\makeatletter\def\blfootnote{\xdef\@thefnmark{}\@footnotetext}\makeatother
\allowdisplaybreaks
\title[Bounded variation, signed measures, Koksma--Hlawka inequality]{\bf Functions of bounded variation, signed measures, and a general Koksma--Hlawka inequality}

\author{Christoph Aistleitner} 
\address{Department of Mathematics and Statistics, Graduate School of Science, Kobe University, 1-1 Rokkodai, Nada-ku, Kobe 657-8501, Japan }
\email{aistleitner@math.tugraz.at}

\author{Josef Dick} 
\address{School of Mathematics and Statistics, University of New South Wales, Sydney NSW 2052, Australia}
\email{josef.dick@unsw.edu.au}

\thanks{The first author is supported by a Schr\"odinger scholarship of the Austrian Research
Foundation (FWF). The second author is supported by an Australian Research Council Queen Elizabeth 2 Fellowship.}


\begin{document}

\begin{abstract}
In this paper we prove a correspondence principle between multivariate functions of bounded variation in the sense of Hardy and Krause and signed measures of finite total variation, which allows us to obtain a simple proof of a generalized Koksma--Hlawka inequality for non-uniform measures. Applications of this inequality to importance sampling in Quasi-Monte Carlo integration and tractability theory are given. Furthermore, we discuss the problem of transforming a low-discrepancy sequence with respect to the uniform measure into a sequence with low discrepancy with respect to a general measure $\mu$, and show the limitations of a method suggested by Chelson.
\end{abstract}

\date{}
\maketitle

\section{Introduction} \label{sec1}

Let $\mathbf{x}_1, \dots, \mathbf{x}_N$ be a set of points in the $d$-dimensional unit cube $[0,1]^d$. The \emph{star-discrepancy} $D_N^*$ of this point set is defined as
\begin{equation} \label{disc}
D_N^*(\mathbf{x}_1, \dots, \mathbf{x}_N) = \sup_{A \in \mathcal{A}^*} \left| \frac{1}{N} \sum_{n=1}^N \mathds{1}_{A} (\mathbf{x}_n) - \lambda(A)\right|.
\end{equation}
Here, and in the sequel, we write $\mathds{1}_A$ for the indicator function of a set $A$, $\lambda$ for the ($d$-dimensional) Lebesgue measure and $\mathcal{A}^*$ for the class of all closed axis-parallel boxes contained in $[0,1]^d$ which have one vertex at the origin. We generally write vectors in bold font. For vectors $\mathbf{a},\mathbf{b}$ we write $\mathbf{a}\leq \mathbf{b}$ and $\mathbf{a}<\mathbf{b}$ if the respective inequalities hold in each coordinate, and we write $[\mathbf{a},\mathbf{b}]$ for the set $\{\mathbf{x}:~\mathbf{a} \leq \mathbf{x} \leq \mathbf{b}\}$. We write $\mathbf{0}$ and $\mathbf{1}$ for the $d$-dimensional vectors $(0,\dots,0)$ and $(1,\dots,1)$, respectively. Since the star-discrepancy $D_N^*$ is the only discrepancy mentioned in this paper, we will use the word ``discrepancy'' synonymously with ``star-discrepancy''.\\

The \emph{Koksma--Hlawka inequality} states that for any function $f$ on $[0,1]^d$ which has bounded variation in the sense of Hardy and Krause and any point set $\mathbf{x}_1, \dots, \mathbf{x}_N \in [0,1]^d$ we have
\begin{equation} \label{kh}
\left|\frac{1}{N} \sum_{n=1}^N f(\mathbf{x}_n) - \int_{[0,1]^d} f(\mathbf{x}) ~d\mathbf{x}\right| \leq \left( \varhk f \right) D_N^*(\mathbf{x}_1, \dots, \mathbf{x}_N).
\end{equation}
A definition of the variation in the sense of Hardy and Krause (denoted by $\varhk$ and abbreviated as \emph{HK-variation} in this paper) is given in Section~\ref{secdef} below. More precisely, by $\varhk$ and HK-variation we mean the variation in the sense of Hardy and Krause \emph{anchored at $\mathbf{1}$} (later in this paper the HK-variation anchored at $\mathbf{0}$ will also play a role; however, the HK-variation anchored at $\mathbf{1}$ is the ``usual'' HK-variation). The one-dimensional version of the inequality~\eqref{kh} was first proved by Koksma~\cite{koksma} in 1942, the multidimensional generalization by Hlawka~\cite{hlawka} in 1961.\\

The Koksma--Hlawka inequality suggests that point sets having small discrepancy can be used for the approximation of the integral of a multivariate function - this observation is one of the cornerstones of the \emph{Quasi-Monte Carlo method} (QMC method) for numerical integration, which uses cleverly designed deterministic points as sampling points of a quadrature rule (as opposed to the \emph{Monte Carlo method}, where randomly generated points are used). Since there exist several constructions of point sets $\mathbf{x}_1,\dots,\mathbf{x}_N$ in $[0,1]^d$ which achieve a discrepancy bounded by
\begin{equation} \label{lowd}
D_N^*(\mathbf{x}_1,\dots,\mathbf{x}_N) \leq c_d (\log N)^{d-1} N^{-1},
\end{equation}
for large $N$ the error estimates in QMC integration can be much better than the (randomized) error of the \emph{Monte Carlo method} (MC method), which is of order $N^{-1/2}$. More information on discrepancy in the context of the previous paragraphs can be found in the monographs~\cite{dpd,dts,kn}. Discrepancy theory in a more general context (geometric, combinatorial, etc.) is described in~\cite{chaz,mat}. A comparison between MC and QMC methods can be found in~\cite{lem}.\\

The QMC method is widely applied to numerical integration problems, for example to the problem of option pricing in financial mathematics. The general idea is that the problem of calculating the expected value of a multidimensional random variable or the problem of calculating an integral over a general domain $\Omega \subset \mathbb{R}^d$ with respect to a general measure $\mu$ can be transformed into an integration problem with respect to the uniform measure on $[0,1]^d$. However, this transformation process is not without its pitfalls, see~\cite{wangtan1,wangtan2}. This is a particularly critical issue as the Koksma--Hlawka inequality is extremely sensitive with respect to the smallest changes of the function $f$, since the slightest deformation can turn a function having small HK-variation into a function of infinite HK-variation. For example, if a smooth function $f$ on a general integration domain $\Omega\subset [0,1]^d$ is simply embedded into $[0,1]^d$ by defining $f(\mathbf{x})=0$ on $[0,1]^d \backslash \Omega$, then clearly $\int_\Omega f(\mathbf{x})~d\mathbf{x} = \int_{[0,1]^d} f(\mathbf{x})~d\mathbf{x}$ but the Koksma--Hlawka inequality is not applicable since the extended function is in general not of bounded HK-variation (unless, roughly speaking, $\Omega$ itself is an axis-parallel box). Similar problems appear if one tries to switch from an integral with respect to a general measure to an integral with respect to the uniform measure. Consequently, it is desirable to find a variant of QMC integration which is directly applicable to integration with respect to general measures (this includes the case of general domains $\Omega \subset [0,1]^d$, by taking a measure which is only supported on $\Omega$). Another motivation for studying such general problems comes from the fact that they are closely related to importance sampling for QMC; see Corollary~\ref{th3} below.\\

Let $\mu$ be a normalized Borel measure\footnote{Throughout this paper we understand that a \emph{measure} is always non-negative, while a \emph{signed measure} may also have negative values.} on $[0,1]^d$. The star-discrepancy with respect to $\mu$ of a point set $\mathbf{x}_1,\dots,\mathbf{x}_N \in[0,1]^d$ is defined as
\begin{equation} \label{discmu}
D_N^*(\mathbf{x}_1, \dots, \mathbf{x}_N;\mu) = \sup_{A \in \mathcal{A}^*} \left| \frac{1}{N} \sum_{n=1}^N \mathds{1}_{A} (\mathbf{x}_n) - \mu(A)\right|.
\end{equation}
Improving results of Beck~\cite{beck}, the authors of the present paper recently showed that for any $\mu$ and any $N$ there exists a point set $\mathbf{x}_1,\dots,\mathbf{x}_N \in[0,1]^d$ for which
\begin{equation} \label{lowd2}
D_N^*(\mathbf{x}_1, \dots, \mathbf{x}_N;\mu) \leq c_d (\log N)^{(3d+1)/2} N^{-1};
\end{equation}
see~\cite{ad}. There is a gap between this upper bound and that for the uniform measure in~\eqref{lowd}, and it is an interesting open problem whether the smallest possible discrepancy with respect to general measures $\mu$ is asymptotically of the same order as the smallest possible discrepancy with respect to the uniform measure. It should be noted that it is also unknown whether~\eqref{lowd} is optimal or if the exponent of the logarithmic term can be further reduced; this is known as the \emph{Grand Open Problem} of discrepancy theory (see~\cite{bil1,bil2}).\\

To show that QMC integration is in principle also possible with respect to general measures, the estimate~\eqref{lowd2} is not sufficient. Additionally one needs a generalized Koksma--Hlawka inequality for non-uniform measures, which is given in Theorem~\ref{th1} below.

\begin{theorem} \label{th1}
Let $f$ be a measurable\footnote{We use the word ``measurable'' in the sense of ``Borel-measurable'', that is in the sense of ``measurable with respect to Borel sets''. It is possible that a function which has bounded HK-variation is always Borel-measurable as well, and that the assumption of $f$ being measurable can be omitted in the statement of the theorem. However, we have not found any evidence of the assertion that bounded HK-variation implies Borel-measurability in the literature.} function on $[0,1]^d$ which has bounded HK-variation. Let $\mu$ be a normalized Borel measure on $[0,1]^d$, and let $\mathbf{x}_1, \dots, \mathbf{x}_N$ be a set of points in $[0,1]^d$. Then
\begin{equation} \label{th1equ}
\left| \frac{1}{N} \sum_{n=1}^N f(\mathbf{x}_n) - \int_{[0,1]^d} f(\mathbf{x}) ~d\mu(\mathbf{x})\right| \leq \left(\varhk f\right) D_N^*(\mathbf{x}_1, \dots, \mathbf{x}_N;\mu).
\end{equation}
\end{theorem}

Theorem~\ref{th1} directly implies the following result, which shows that the method of importance sampling can be used for Quasi-Monte Carlo integration. 
\begin{corollary} \label{th3}
Let $f$ be a measurable function on $[0,1]^d$, and let $g$ be the density of a normalized Borel measure $\mu_g$ on $[0,1]^d$.  Assume further that $f/g$ has bounded HK-variation, and that $g(\mathbf{x})>0$ for all $\mathbf{x} \in [0,1]^d$. Let $\mathbf{x}_1, \dots, \mathbf{x}_N$ be a set of points in $[0,1]^d$. Then
\begin{equation} \label{th3equ}
\left| \frac{1}{N} \sum_{n=1}^N \frac{f(\mathbf{x}_n)}{g(\mathbf{x}_n)} - \int_{[0,1]^d} f(\mathbf{x}) ~d\mathbf{x} \right| \leq \left(\varhk \left(\frac{f}{g}\right)\right) D_N^*(\mathbf{x}_1, \dots, \mathbf{x}_N;\mu_g).
\end{equation}
\end{corollary}
The idea of importance sampling is to find a function $g$ for which the HK-variation of $f/g$ is significantly smaller than that of $f$; in this case the error bound in~\eqref{th3equ} can be much better than that in the standard Koksma--Hlawka inequality~\eqref{kh}.\\

Corollary~\ref{th3} was obtained by Chelson~\cite{chelson} in his PhD thesis, which was published in 1976; it is stated there with the incorrect conclusion that on the right-hand side of~\eqref{th3equ} one can take the discrepancy with respect to the \emph{uniform} measure of a point set which is related to $\mathbf{x}_1, \dots, \mathbf{x}_N$ by a simple transformation, instead of the discrepancy of $\mathbf{x}_1, \dots, \mathbf{x}_N$ with respect to the measure induced by $g$. Chelson's result and its correct and incorrect parts are described in detail in Section~\ref{sectrans}. Chelson's result is formulated in the language of Corollary~\ref{th3}, which can only be sensibly stated with the assumption that $g$ is the density of a measure; accordingly, a variant of Theorem~\ref{th1} can only be deduced from Chelson's formulation with the additional assumption that $\mu$ possesses a density, and not for general~$\mu$.\\

Because of the issues mentioned in the previous paragraph, the main impulse for writing this paper was to discuss Chelson's result and method, and to give a correct proof of Theorem~\ref{th1} without assuming that the measure $\mu$ possesses a density. However, when the present manuscript was almost finished, we coincidentally found out that such a result had already been obtained by G\"otz~\cite{gotz}. G\"otz's paper was published in 2002, but apparently it was almost completely overlooked until now; we only found it casually noted in a short survey article of Niederreiter~\cite{nied}. Apparently G\"otz did not know of Chelson's result. It should be noted that despite the remarks concerning Chelson's result in the previous paragraph, his method of proof is in principle correct, and could be modified to give a correct proof of Corollary~\ref{th3}. Chelson's and G\"otz's results are both proved using the same method which is usually used for proving the standard Koksma--Hlawka inequality, namely Abel 
partial 
summation. Our proof for Theorem~\ref{th1} is simpler; it can be seen as an application of a partial integration formula for the Stieltjes integral, which, however, is nothing other than the continuous analogue of the Abel summation formula. The proof is based on a correspondence principle between functions of bounded HK-variation and signed measures of finite total variation, which is of some interest in its own right (see Theorem~\ref{thbv} below). Together with recent new results on the existence of low-discrepancy point sets with respect to general measures $\mu$, Theorem~\ref{th1} implies strong convergence results for QMC integration with respect to general measures (see Corollaries~\ref{co1} and~\ref{co2} below).\\

A result somewhat similar to Theorem~\ref{th1} in the case when the measure $\mu$ is the uniform measure on the unit simplex was obtained in~\cite{basu,pill,pill2}. A result similar to Theorem~\ref{th1} in the special case when $\mu$ is the uniform measure on a set $\Omega \subset [0,1]^d$ (or, more generally, for bounded $\Omega \subset \mathbb{R}^d$) has been obtained recently by Brandolini \emph{et al.}~\cite{brand}; however, their error estimate contains multiplicative factors which depend exponentially on the dimension, and which accordingly spoil all tractability results (the case of star-discrepancy and HK-variation is the special case $p=1$ and $q=\infty$ of the more general result in their paper). Another result of Brandolini \emph{et al.}~\cite{brand2} gives a general Koksma--Hlawka inequality for the uniform measure on compact parallelepipeds or simplices.\\

The following theorem establishes the existence of the Jordan decomposition of a multivariate function of bounded HK-variation. It is a generalization of the well-known Jordan decomposition theorem for functions of bounded variation in the one-dimensional case (see for example~\cite[{\S 12, Section III}]{yeh}). The key ingredient in its proof is a decomposition theorem of Leonov~\cite[Theorem 3]{leon}. The statement of the theorem uses the notion of a \emph{completely monotone function}, which is defined in Section~\ref{secdef} below. It also uses the notion of the \emph{HK-variation anchored at $\mathbf{0}$}, which is also defined in Section~\ref{secdef}, and which is denoted by \emph{HK$\mathbf{0}$-variation} and $\varhkn$ in this paper. The HK$\mathbf{0}$-variation of a function is in general different from the ``usual'' HK-variation anchored at $\mathbf{1}$. However, a function which has bounded HK-variation also has bounded HK$\mathbf{0}$-variation, and vice versa (see Lemma~\ref{lemmahk0} below). Thus 
in the assumptions of Theorems~\ref{th1}-\ref{thbv} and throughout this paper the phrase ``$f$ has bounded HK-variation'' could be replaced by ``$f$ has bounded HK$\mathbf{0}$-variation''. However, the variations $\varhk$ and $\varhkn$ must \emph{not} be simply exchanged in the conclusions of the respective theorems.

\begin{theorem} \label{thbv2}
Let $f$ be a function on $[0,1]^d$ which has bounded HK-variation. Then there exist two uniquely determined completely monotone functions $f^+$ and $f^-$ on $[0,1]^d$ such that $f^+(\mathbf{0})=f^-(\mathbf{0})=0$ and
$$
f(\mathbf{x}) = f(\mathbf{0}) + f^+(\mathbf{x}) - f^-(\mathbf{x}), \qquad \mathbf{x} \in [0,1]^d,
$$
and
\begin{equation} \label{sumofv}
\varhkn f = \varhkn f^+ + \varhkn f^-.
\end{equation}
\end{theorem}

We call the unique decomposition $f=f^+-f^-$ having the properties described in Theorem~\ref{thbv2} the \emph{Jordan decomposition} of the function $f$. We note that using relation~\eqref{fgrel} below it is easy to obtain a variant of Theorem~\ref{thbv2} for the HK-variation instead of the HK$\mathbf{0}$-variation.\\

The following theorem shows, simply speaking, that any right-continuous function of bounded HK-variation defines a finite signed measure and vice versa, and that the HK$\mathbf{0}$-variation of the function and the total variation of the signed measure coincide. The Jordan decomposition of a signed measure and the total variation of a signed measure, denoted by $\vart$, are defined in Section~\ref{secdef} below.

\begin{theorem} \label{thbv}
\begin{enumerate}[label=\textbf{(\alph*)}]
\item Let $f$ be a right-continuous\footnote{We call a multivariate function right-continuous if it is coordinatewise right-continuous in each coordinate, at every point.} function on $[0,1]^d$ which has bounded HK-variation. Then there exists a unique signed Borel measure $\nu$ on $[0,1]^d$ for which
\begin{equation} \label{ths1}
f(\mathbf{x}) = \nu([\mathbf{0},\mathbf{x}]), \qquad \mathbf{x} \in [0,1]^d.
\end{equation}
Then we have
\begin{equation} \label{ths2}
\vart \nu = \varhkn f + |f(\mathbf{0})|.
\end{equation}
Furthermore, if $f(\mathbf{x})=f(\mathbf{0}) + f^+(\mathbf{x}) - f^- (\mathbf{x})$ is the Jordan decomposition of $f$, and $\nu=\nu^+-\nu^-$ is the Jordan decomposition\footnote{Note that contrary to the Jordan decomposition of a multivariate function of bounded variation, which we have not found in the literature and whose definition is given and whose existence is established by Theorem~\ref{thbv2}, the Jordan decomposition of a signed measure is a well-known concept in measure theory; in particular, the Jordan decomposition of a signed measure always exists and is unique. See Section~\ref{secdef} for details.} of $\nu$, then 
\begin{equation} \label{ths3}
f^+(\mathbf{x})=\nu^+([\mathbf{0},\mathbf{x}]\backslash\{\mathbf{0}\}) \qquad \textrm{and} \qquad f^-(\mathbf{x})=\nu^-([\mathbf{0},\mathbf{x}]\backslash\{\mathbf{0}\}), \qquad \mathbf{x} \in [0,1]^d.
\end{equation}
\item Let $\nu$ be a finite signed Borel measure on $[0,1]^d$. Then there exists a unique right-continuous function of bounded HK-variation $f$ on $[0,1]^d$ for which~\eqref{ths1} and~\eqref{ths2} hold. Furthermore, if $f(\mathbf{x})=f(\mathbf{0}) + f^+(\mathbf{x})-f^-(\mathbf{x})$ is the Jordan decomposition of $f$ and $\nu=\nu^+-\nu^-$ is the Jordan decomposition of $\nu$, then~\eqref{ths3} holds.
\end{enumerate}
\end{theorem}

This connection between functions of bounded HK-variation and finite signed measures is quite natural, and has probably been observed and used in a less specific form before. For example, this is the same mechanism by which a multidimensional additive set-function of bounded variation defines a finite signed measure and a corresponding multidimensional Stieltjes-integral, as described in Chapter 3 of Stanis{\l}aw Saks'~\cite{saks} classical monograph on the \emph{Theory of the Integral}. In essence, this is also the same principle which has been used by Zaremba~\cite{zaremba} for his proof of the Koksma--Hlawka inequality by multidimensional Abel summation, which is now the standard method. However, the decomposition in Theorem~\ref{thbv2} and the correspondence between functions of bounded variation and finite signed measures in Theorem~\ref{thbv}, and in particular the relation between the HK$\mathbf{0}$-variation of the function and the total variation of the corresponding signed measure, are far from 
being self-evident, and we have not found any explicit statement like Theorem~\ref{thbv2} or Theorem~\ref{thbv} in the literature.\\

A somewhat vague connection between functions of bounded HK-variation and signed measures is casually noted in~\cite{boul}. A new notion of bounded variation of a function (which has later been called \emph{bounded variation in the measure sense}), which is defined in terms of a signed measure corresponding to a function, is defined in~\cite{bl}, and is used there to prove a general Koksma--Hlawka inequality; however, it is not stated which functions are of bounded variation in this sense. A connection between functions of bounded HK-variation and functions of bounded variation in the measure sense (significantly weaker than our Theorem~\ref{thbv}) is mentioned as a ``Proposition'' (without proof) in~\cite{xiao}, with reference to an unpublished manuscript. The same result is stated as a ``Theorem'' (without proof) in~\cite{okten}, and then, by the same author and several years later, in~\cite{okter} as an ``unproven conjecture''. As our Theorem~\ref{thbv} shows the notion of bounded variation in the measure sense is superfluous, since it 
coincides with the notion of bounded HK-variation (aside from continuity issues).\\

Finally, we want to state two consequences of Theorem~\ref{th1} and the recent results obtained by the authors in~\cite{ad}. The first result shows that QMC integration with respect to general measures is possible with a convergence rate for the error which is almost the same as in the case of the uniform measure; the second is a \emph{tractability} result, which states that QMC integration is possible with a moderate number of sampling points in comparison with the dimension, just as in the case of the uniform measure.

\begin{corollary} \label{co1}
For any normalized Borel measure $\mu$ on $[0,1]^d$ and any $N \geq 1$ there exist points $\mathbf{x}_1,\dots,\mathbf{x}_N \in [0,1]^d$ for which
$$
\sup_f \left|\frac{1}{N} \sum_{n=1}^N f(\mathbf{x_n}) - \int_{[0,1]^d} f(\mathbf{x}) ~d\mu(\mathbf{x}) \right| \leq 63 \sqrt{d} \frac{ \left(2 + \log_2 N \right)^{(3d+1)/2}}{N},
$$
where the supremum is taken over all measurable functions $f$ on $[0,1]^d$ which satisfy $\varhk f \leq 1$.
\end{corollary}

\begin{corollary} \label{co2}
Let $\ve>0$ be given. Then for any normalized Borel measure $\mu$ on $[0,1]^d$ there exists a point set $\mathbf{x}_1,\dots,\mathbf{x}_N \in [0,1]^d$ such that
\begin{equation} \label{con}
N \leq 2^{26} d \ve^{-2}
\end{equation}
and
\begin{equation} \label{thtrequ}
\sup_f \left|\frac{1}{N} \sum_{n=1}^N f(\mathbf{x_n}) - \int_{[0,1]^d} f(\mathbf{x}) ~d\mu(\mathbf{x}) \right| \leq \ve,
\end{equation}
where the supremum is taken over all measurable functions $f$ on $[0,1]^d$ which satisfy $\varhk f \leq 1$.
\end{corollary}

The corollaries follow directly from Theorem~\ref{th1} together with~\cite[Theorem 1]{ad} and~\cite[Corollary 1]{ad}, respectively. Corollary~\ref{co2} implies that the problem of integrating $d$-dimensional functions whose HK-variation is uniformly bounded with respect to \emph{any} normalized measure is polynomially tractable. For more information on tractability see the three volumes on \emph{Tractability of Multivariate Problems} by Novak and Wo{\'z}niakowski~\cite{n1,n2,n3}. In the case of the function class being restricted to indicator functions of sets $A \in \mathcal{A}^*$ (that is, in the case of the left-hand side of~\eqref{thtrequ} being the star-discrepancy with respect to $\mu$) Corollary~\ref{co2} was proved in~\cite{hnww} (without an effective value for the constant in~\eqref{con}). It should be noted that the results in~\cite{ad} are pure existence results, and that the problem of constructing point sets satisfying the conclusions of Corollary~\ref{co1} and~\ref{co2} is completely open. This 
problem is shortly addressed in Section~\ref{sectrans}.\\

The outline of the remaining part of this paper is as follows. Section~\ref{secdef} contains all necessary definitions, and some basic properties of the concepts needed for our proofs. In Section~\ref{thbv} the proof of Theorem~\ref{thbv2} and~\ref{thbv} is given, as well as the proof of a lemma establishing a connection between the HK-variation and the HK$\mathbf{0}$-variation (see Lemma~\ref{lemmahk0} below). In Section~\ref{seckh} we deduce Theorem~\ref{th1} from Theorem~\ref{thbv}. In Section~\ref{sectrans} we discuss Chelson's result, which is formulated together with a transformation process which supposedly transforms low-discrepancy point sets with respect to $\lambda$ into low-discrepancy point sets with respect to a general measure $\mu$. We show why this transformation does not have the alleged properties, and that it does actually work when the measure $\mu$ is of product type.

\section{Definitions and basic properties} \label{secdef}

To define the variation in the sense of Hardy and Krause, we follow the exposition in~\cite{leon}; however, additionally we pay attention to the fact that we can put the anchor either to the lower left corner $\mathbf{0}$ (as in~\cite{leon}) or to the upper right corner $\mathbf{1}$ (as is usually done in the context of discrepancy theory; see for example~\cite{kn}). Subsequently, we introduce completely monotone functions, following~\cite[Section 3]{leon}. We note that a comprehensive survey on HK-variation and its properties can be found in~\cite{owen}.\\

Let $f(\mathbf{x})$ be a function on $[0,1]^d$. Let $\mathbf{a}=(a_1,\dots,a_d)$ and $\mathbf{b}=(b_1,\dots,b_d)$ be elements of $[0,1]^d$ such that $\mathbf{a}<\mathbf{b}$. We introduce the $d$-dimensional difference operator $\Delta^{(d)}$, which assigns to the axis-parallel box $A = [\mathbf{a},\mathbf{b}]$ a $d$-dimensional \emph{quasi-volume}
\begin{equation} \label{quasi}
\Delta^{(d)} (f;A) = \sum_{j_1=0}^1 \dots \sum_{j_d=0}^1 (-1)^{j_1+\dots+j_d} f(b_1 + j_1 (a_1 - b_1), \dots,b_d + j_d (a_d - b_d)).
\end{equation}
For $s=1, \dots, d$, let 
$$
0 = 	x_0^{(s)} < x_1^{(s)} < \dots < x_{m_s}^{(s)} =1 
$$
be a partition of $[0,1]$, and let $\mathcal{P}$ be the partition of $[0,1]^d$ which is given by
\begin{equation} \label{part}
\mathcal{P}=\left\{ \left[x_{l_1}^{(1)},x_{l_1+1}^{(1)}\right] \times \dots \times \left[x_{l_d}^{(d)},x_{l_d+1}^{(d)}\right], \qquad l_s=0, \dots, m_s-1,~s=1, \dots, d\right\}.
\end{equation}
Then the \emph{variation of $f$ on $[0,1]^d$ in the sense of Vitali} is given by
\begin{equation} \label{vitali}
V^{(d)}(f;[0,1]^d) = \sup_{\mathcal{P}} \sum_{A \in \mathcal{P}} \left| \Delta^{(d)} (f;A) \right|,
\end{equation}
where the supremum is extended over all partitions of $[0,1]^d$ into axis-parallel boxes generated by $d$ one-dimensional partitions of $[0,1]$, as in~\eqref{part}. For $1 \leq s \leq d$ and $1 \leq i_1 < \dots < i_s \leq d$, let $V^{(s)}(f;i_1, \dots, i_s;[0,1]^d)$ denote the $s$-dimensional variation in the sense of Vitali of the restriction of $f$ to the face
\begin{equation} \label{face1}
U_d^{(i_1,\dots,i_s)} = \left\{(x_1, \dots,x_d) \in [0,1]^d:~x_j = 1~\textrm{for all} ~j \neq i_1, \dots, i_s\right\}
\end{equation}
of $[0,1]^d$. Then the \emph{variation of $f$ on $[0,1]^d$ in the sense of Hardy and Krause anchored at~$\mathbf{1}$}, abbreviated by HK-variation, is given by
\begin{equation} \label{hkdef}
\varhk (f;[0,1]^d) = \sum_{s=1}^d ~ \sum_{1 \leq i_1 < \dots < i_s \leq d} V^{(s)}\left(f;i_1, \dots, i_s;[0,1]^d\right).
\end{equation}

Note that for the definition of the HK-variation in~\eqref{hkdef}, we add the $d$-dimensional variation in the sense of Vitali plus the variation in the sense of Vitali on all lower-dimensional faces of $[0,1]^d$ which are adjacent to $\mathbf{1}$. For the HK$\mathbf{0}$-variation, we take instead the sum over those lower-dimensional faces which are adjacent to $\mathbf{0}$. More precisely, let $V^{(s;\mathbf{0})}(f;i_1, \dots, i_s;[0,1]^d)$ denote the $s$-dimensional variation in the sense of Vitali of the restriction of $f$ to the face
\begin{equation*} 
W_d^{(i_1,\dots,i_s)} = \left\{(x_1, \dots,x_d) \in [0,1]^d:~x_j = 0~\textrm{for all} ~j \neq i_1, \dots, i_s\right\}.
\end{equation*}
Then the \emph{variation of $f$ on $[0,1]^d$ in the sense of Hardy and Krause anchored at~$\mathbf{0}$}, abbreviated by HK$\mathbf{0}$-variation, is given by
\begin{equation} \label{hkdef0}
\varhkn (f;[0,1]^d) = \sum_{s=1}^d ~ \sum_{1 \leq i_1 < \dots < i_s \leq d} V^{(s;\mathbf{0})}\left(f;i_1, \dots, i_s;[0,1]^d \right).
\end{equation}
For any $\mathbf{a} \in [0,1]^d,~\mathbf{a} \neq \mathbf{0},$ we can define the variation in the sense of Vitali and the HK$\mathbf{0}$-variation of $f$ on $[\mathbf{0},\mathbf{a}]$ in a similar way as above, by considering decompositions of $[\mathbf{0},\mathbf{a}]$ into axis-parallel boxes instead of decompositions of $[0,1]^d$, and again taking a sum over all lower-dimensional faces adjacent to $\mathbf{0}$ as in~\eqref{hkdef0}. For notational convenience we also define $\varhkn(f;[\mathbf{0},\mathbf{0}])=0$. Throughout this paper, we simply write $\varhk f$ and $\varhkn f$ for $\varhk (f;[0,1]^d)$ and $\varhkn (f;[0,1]^d)$, respectively.\\

The HK-variation and the HK$\mathbf{0}$-variation of a function are in general different; for example, the indicator function $f$ of the closed axis-parallel box stretching from the point $(1/2,\dots,1/2)$ to $\mathbf{1}$ has HK-variation $2^d-1$, but HK$\mathbf{0}$-variation only 1. This difference reflects the fact that on the one hand the function $f$ can be written as the sum/difference of no less than $2^d-1$ indicator functions of axis-parallel boxes which have one vertex at the origin (which affects the error term in the Koksma--Hlawka inequality for $f$ in Theorem~\ref{th1}), but on the other hand $f$ is the distribution function of a measure whose total mass is only~1 (namely the Dirac measure centered at $(1/2,\dots,1/2)$; consequently this version of the variation appears in Theorem~\ref{thbv}). This example represents the ``worst case'': we will show in Lemma~\ref{lemmahk0} below that we always have $\varhk \leq (2^d-1)\varhkn$ and vice versa.\\

We note that by a simple mirroring argument for any function of bounded HK-variation $f$ we have
\begin{equation} \label{fgrel}
\varhk f = \varhkn g, \qquad \textrm{where $g$ is defined by $g(\mathbf{x})=f(\mathbf{1}-\mathbf{x})$,\quad $\mathbf{x} \in [0,1]^d$;}
\end{equation}
this relation will be needed in the proof of Theorem~\ref{th1}, and explains why the HK$\mathbf{0}$-variation turns into the HK-variation when Theorem~\ref{thbv} is used to prove Theorem~\ref{th1} (see Section~\ref{seckh}).\\

We will also need the following lemma. 

\begin{lemma}\label{lemmatri}
Let $f$ and $g$ be functions on $[0,1]^d$ which have bounded HK$\mathbf{0}$-variation. Then for any $\mathbf{a},\mathbf{b} \in [0,1]^d,~\mathbf{a} \leq \mathbf{b}$, we have
\begin{eqnarray*}
& & \varhkn (f+g;[\mathbf{0},\mathbf{b}]) - \varhkn (f+g;[\mathbf{0},\mathbf{a}]) \\
& \leq & \varhkn (f;[\mathbf{0},\mathbf{b}]) + \varhkn (g;[\mathbf{0},\mathbf{b}]) - \varhkn (f;[\mathbf{0},\mathbf{a}]) - \varhkn (g;[\mathbf{0},\mathbf{a}]).
\end{eqnarray*}
\end{lemma}

Note that as a special case of the lemma (for $\mathbf{a}=\mathbf{0}$) we have 
\begin{equation} \label{trian}
\varhkn (f+g;[\mathbf{0},\mathbf{b}]) \leq \varhkn (f;[\mathbf{0},\mathbf{b}]) + \varhkn (g;[\mathbf{0},\mathbf{b}]),
\end{equation} 
which is just the triangle inequality for the HK$\mathbf{0}$-variation. In a similar way we could prove the (well-known) triangle inequality for the HK-variation: we have 
\begin{equation} \label{trian2}
\varhk (f+g) \leq \varhk f + \varhk g.
\end{equation}

\begin{proof}[Proof of Lemma~\ref{lemmatri}]
There exist a number $m$ and axis-parallel boxes
$$
\left[\mathbf{a}_i,\mathbf{b}_i\right], \qquad i=1, \dots, m,
$$
such that $[\mathbf{a}_1,\mathbf{b}_1] = [\mathbf{0},\mathbf{a}]$ and such that
$$
\bigcup_{i=1}^m [\mathbf{a}_i,\mathbf{b}_i] = [\mathbf{0},\mathbf{b}] \qquad \textrm{and} \qquad [\mathbf{a}_i,\mathbf{b}_i) \cap [\mathbf{a}_j,\mathbf{b}_j) = \emptyset ~\textrm{for $i \neq j$}.
$$
The system $[\mathbf{a}_i,\mathbf{b}_i], ~1 \leq i \leq m$ is called a \emph{split} of $[\mathbf{0},\mathbf{b}]$, and for any function $h$ which has bounded variation on $[0,1]^d$ we have
\begin{equation} \label{grea4}
V^{(d)}(h;[\mathbf{0},\mathbf{b}])  = \sum_{i=1}^m V^{(d)} (h;[\mathbf{a}_i,\mathbf{b}_i])
\end{equation}
(this is stated, for example, in~\cite[Lemma 1]{owen}). Thus by the triangle inequality for the variation in the sense of Vitali (which follows directly from the ordinary triangle inequality for real numbers) we have
\begin{eqnarray*}
& & V^{(d)}(f+g;[\mathbf{0},\mathbf{b}]) - V^{(d)}(f+g;[\mathbf{0},\mathbf{a}]) \\
& = & \sum_{i=2}^m V^{(d)} (f+g;[\mathbf{a}_i,\mathbf{b}_i]) \\
& \leq & \sum_{i=2}^m \left(V^{(d)} (f;[\mathbf{a}_i,\mathbf{b}_i]) + V^{(d)} (g;[\mathbf{a}_i,\mathbf{b}_i])\right) \\
& = & V^{(d)}(f;[\mathbf{0},\mathbf{b}]) - V^{(d)}(f;[\mathbf{0},\mathbf{a}]) + V^{(d)}(g;[\mathbf{0},\mathbf{b}]) - V^{(d)}(g;[\mathbf{0},\mathbf{a}]).
\end{eqnarray*}
Similar inequalities hold for the variation in the sense of Vitali on all lower-dimensional faces of $[\mathbf{0},\mathbf{b}]$ adjacent to $\mathbf{0}$. Since the HK$\mathbf{0}$-variation is defined as the sum over these variations in the sense of Vitali, and since the requested inequality holds in each summand, we obtain the conclusion of the lemma.
\end{proof}

The following lemma establishes the connection between HK-variation and HK$\mathbf{0}$-variation, which was already announced before the statement of Theorem~\ref{thbv2}. Its proof relies upon Theorem~\ref{thbv}, and will be given at the end of Section~\ref{secproofsbv}.

\begin{lemma} \label{lemmahk0}
Let $f$ be a function on $[0,1]^d$ which has bounded HK$\mathbf{0}$-variation. Then $f$ has bounded HK-variation as well, and 
$$
\varhk f \leq \left(2^d-1\right) \varhkn f.
$$
The same statement holds if the roles of HK$\mathbf{0}$-variation and HK-variation are interchanged.
\end{lemma}

A function $h$ on $[0,1]^d$ is called \emph{completely monotone} if for any closed axis-parallel box $A \subset [0,1]^d$ of arbitrary dimension $s$ (where $1 \leq s \leq d$) its $s$-dimensional quasi-volume $\Delta^{(s)}$ generated by the function $h$ is non-negative (other words which are used for this property are \emph{quasi-monotone}, \emph{monotonely monotone} and \emph{entirely monotone}). A function of bounded HK-variation can be split into the difference of two completely monotone functions; for the two-dimensional case this is mentioned in~\cite{ac}, where it is attributed to Hobson~\cite{hobson}. The following result of Leonov~\cite{leon} shows a way for obtaining such a decomposition. 

\begin{lemma}[{\cite[Theorem 3]{leon}}] \label{lemmaleo}
Let $f(\mathbf{x})$ be a function on $[0,1]^d$, which has bounded HK$\mathbf{0}$-variation. Then the functions 
\begin{equation} \label{equleo}
f_1 (\mathbf{x}) = \varhkn(f;[\mathbf{0},\mathbf{x}]) \qquad \textrm{and} \qquad f_2 (\mathbf{x}) = f_1(\mathbf{x}) - f(\mathbf{x})
\end{equation}
are completely monotone, and
$$
f(\mathbf{x}) = f_1(\mathbf{x}) - f_2(\mathbf{x}), \qquad \mathbf{x} \in [0,1]^d.
$$
\end{lemma}

Note that the decomposition into two completely monotone functions is not unique. If $f$ has bounded HK$\mathbf{0}$-variation and $g,h$ are completely monotone functions such that $f=g-h$, then for any completely monotone function $r$ the two functions $f+r$ and $g+r$ also form a decomposition of $f$ as the difference of two completely monotone functions. Thus the decomposition given in Lemma~\ref{lemmaleo} is just one out of many possible decompositions of $f$; in particular, it is \emph{not} the outstanding decomposition which is mentioned in Theorem~\ref{thbv2}.\\

For a completely monotone function $h$ we have $\varhkn(h;[\mathbf{0},\mathbf{x}]) = h(\mathbf{x})-h(\mathbf{0})$ (as noted after~\cite[Definition 2]{leon}, this follows from Equations (6), (7) and Theorem 1 of~\cite{leon}). Consequently, for the functions $f_1,f_2$ in Lemma~\ref{lemmaleo} we have $\varhkn f_1 = \varhkn f$ and $\varhkn f_2 = \varhkn f-f(\mathbf{1})+f(\mathbf{0})$, which implies that both functions $f_1,f_2$ are of bounded HK$\mathbf{0}$-variation (and thus, by Lemma~\ref{lemmahk0}, also of bounded HK-variation).\\

A \emph{signed measure} is a measure which is also allowed to have negative values. A formal definition can be found for example in~\cite[Chapter 10]{yeh}. By the Jordan Decomposition Theorem (see for example~\cite[Theorem 10.21]{yeh}) any signed measure $\nu$ on a measurable space $(\Omega,\mathscr{A})$ can be uniquely decomposed into a ``positive'' and a ``negative'' part which are orthogonal to each other. More precisely, there exist measures $\nu^+$ and $\nu^-$ such that $\nu^+ \perp \nu^-$ and $\nu = \nu^+ - \nu^-$. Here $\nu^+ \perp \nu^-$ means that there exist two sets $C^+,C^- \in \mathscr{A}$ such that $\Omega = C^+ \cup C^-$ and $C^+ \cap C^- = \emptyset$, and such that $\nu^-(C^+)=0$ and $\nu^+(C^-)=0$. Furthermore, at least one of the measures $\nu^+,\nu^-$ is finite; as a consequence, if $\nu$ is finite, then both $\nu^+$ and $\nu^-$ must be finite. The pair $(\nu^+,\nu^-)$ is called the \emph{Jordan decomposition} of $\nu$. The measure $|\nu|=\nu^+ + \nu^-$ is called the variation measure of 
$\nu$, and the quantity $\vart \nu = |\nu|(\Omega)=\nu^+(\Omega)+\nu^-(\Omega)$ is called the total variation of $\nu$.\\

\section{Functions of bounded variation and signed measures:~Proof of Theorem~\ref{thbv2}, Theorem~\ref{thbv} and Lemma~\ref{lemmahk0}} \label{secproofsbv}

In the proofs of Theorem~\ref{thbv2} and Theorem~\ref{thbv} below, only the variation anchored at $\mathbf{0}$ (that is, the HK$\mathbf{0}$-variation) plays a role. Subsequently, in Lemma~\ref{lemmahk0}, the connection between HK$\mathbf{0}$-variation and HK-variation is established.

\begin{proof}[Proof of Theorem~\ref{thbv2}]
Let $f$ be a function on $[0,1]^d$ which has bounded HK$\mathbf{0}$-variation, and let $f_1,f_2$ be the functions defined in Lemma~\ref{lemmaleo}. These functions do not provide the desired decomposition; in fact, we have
\begin{eqnarray*}
\varhkn f_1 + \varhkn f_2 & = & f_1(\mathbf{1})-f_1(\mathbf{0}) + f_2(\mathbf{1}) - f_2 (\mathbf{0}) \\
& = & \varhkn f + \varhkn f - f(\mathbf{1}) + f(\mathbf{0})\\
& = & 2 \varhkn f - f(\mathbf{1}) + f(\mathbf{0}),
\end{eqnarray*}
and while for any function $f$ of bounded HK$\mathbf{0}$-variation we have
\begin{equation} \label{dreie}
f(\mathbf{1})-f(\mathbf{0}) \leq \varhkn f,
\end{equation}
there is in general no equality in~\eqref{dreie}. Consequently, the sum of the variations of the functions $f_1$ and $f_2$ from Lemma~\ref{lemmaleo} is in general larger than the variation of $f$.\\

Instead, we define the functions
\begin{eqnarray}
f^+(\mathbf{x}) & = & \frac{1}{2} \left(\varhkn(f;[\mathbf{0},\mathbf{x}]) + f(\mathbf{x}) - f(\mathbf{0}) \right),  \label{f+-} \\
f^-(\mathbf{x}) & = & \frac{1}{2} \left(\varhkn(f;[\mathbf{0},\mathbf{x}]) - f(\mathbf{x}) + f(\mathbf{0}) \right) \label{f+-2}
\end{eqnarray}
for $\mathbf{x} \in [0,1]^d$. Then obviously we have $f(\mathbf{x})=f(\mathbf{0})+f^+(\mathbf{x}) - f^-(\mathbf{x})$ for every $\mathbf{x}$. Furthermore, since the function $f_2$ in Lemma~\ref{lemmaleo} is completely monotone, the same is true for the function $f^-$ in~\eqref{f+-2}. Now set $g=-f$. Then it is easily seen that for any $\mathbf{x}$ we have $\varhkn(g;[\mathbf{0},\mathbf{x}])=\varhkn(f;[\mathbf{0},\mathbf{x}])$. Applying Lemma~\ref{lemmaleo} to $g$ we see that the function 
$$
\varhkn(g;[\mathbf{0},\mathbf{x}])-g(\mathbf{x}) = \varhkn(f;[\mathbf{0},\mathbf{x}])+f(\mathbf{x}) = 2 f^+(\mathbf{x}) + f(\mathbf{0})
$$
is completely monotone, which implies that the function $f^+$ is also completely monotone. Thus both $f^+$ and $f^-$ are completely monotone, and we have
\begin{eqnarray*}
\varhkn f^+ & = & f^+(\mathbf{1})-f^+(\mathbf{0}) \\
& = & \frac{1}{2} (\varhkn f + f(\mathbf{1})-f(\mathbf{0}))
\end{eqnarray*}
and similarly
\begin{eqnarray*}
\varhkn f^- = \frac{1}{2} (\varhkn f - f(\mathbf{1})+f(\mathbf{0})),
\end{eqnarray*}
which proves that
$$
\varhkn f = \varhkn f^+ + \varhkn f^-.
$$
It is easily seen that $f^+(\mathbf{0})=f^-(\mathbf{0})=0$, and thus the functions $f^+$ and $f^-$ from~\eqref{f+-} and~\eqref{f+-2} have the properties requested in Theorem~\ref{thbv2}.\\

It remains to show that this decomposition is the only one satisfying the statement of Theorem~\ref{thbv2}. Thus, suppose that $g^+$ and $g^-$ are two completely monotone functions such that $f(\mathbf{x})=f(\mathbf{0}) + g^+ (\mathbf{x}) -g^-(\mathbf{x})$ for every $\mathbf{x}$ and $g^+(\mathbf{0})=g^-(\mathbf{0})=0$. Then for every $\mathbf{x}$ we have
\begin{eqnarray}
g^+(\mathbf{x})+g^-(\mathbf{x}) & = & \varhkn (g^+;[\mathbf{0},\mathbf{x}]) +\varhkn (g^-;[\mathbf{0},\mathbf{x}]) \nonumber\\
& \geq & \varhkn (f;[\mathbf{0},\mathbf{x}]) \label{triin} \\
& = & f^+(\mathbf{x}) + f^-(\mathbf{x}), \nonumber
\end{eqnarray}
where~\eqref{triin} follows from the triangle inequality for the HK$\mathbf{0}$-variation, that is, from~\eqref{trian}. By adding $f^+-f^-=g^+-g^-$ to each line of this inequality we obtain
\begin{equation} \label{grea}
g^+(\mathbf{x}) \geq f^+(\mathbf{x}), \qquad \mathbf{x} \in [0,1]^d,
\end{equation}
and by subtracting the same quantity from each line we similarly get
\begin{equation} \label{grea2}
g^-(\mathbf{x}) \geq f^-(\mathbf{x}), \qquad \mathbf{x} \in [0,1]^d.
\end{equation}
Suppose that there exists a point $\mathbf{x} \in [0,1]^d$ such that $g^+(\mathbf{x}) \neq f^+(\mathbf{x})$. By~\eqref{grea} this implies $g^+(\mathbf{x}) > f^+(\mathbf{x})$, which together with~\eqref{grea2} and the assumption that $g^+(\mathbf{0})=g^-(\mathbf{0})=0$ also implies
\begin{equation} \label{grea3}
\varhkn (g^+;[\mathbf{0},\mathbf{x}]) + \varhkn (g^-;[\mathbf{0},\mathbf{x}]) = g^+(\mathbf{x}) + g^-(\mathbf{x}) > \varhkn (f;[\mathbf{0},\mathbf{x}]).
\end{equation}
By Lemma~\ref{lemmatri} we have 
\begin{eqnarray*}
& & \varhkn(g^+;[\mathbf{0},\mathbf{1}])+\varhkn(g^-;[\mathbf{0},\mathbf{1}])-\varhkn(g^+;[\mathbf{0},\mathbf{x}])-\varhkn(g^-;[\mathbf{0},\mathbf{x}]) \\
& \geq & \varhkn(f;[\mathbf{0},\mathbf{1}]) - \varhkn(f;[\mathbf{0},\mathbf{x}]).
\end{eqnarray*}
Combining this with~\eqref{grea3} we have
$$
\varhkn g^+ +\varhkn g^- >\varhkn f.
$$
Thus the decomposition of $f$ into $g^+$ and $g^-$ violates~\eqref{sumofv}, which means that it does not have the properties requested in Theorem~\ref{thbv2}. Consequently, the decomposition of $f$ described in Theorem~\ref{thbv2} is unique.
\end{proof}

We note that the functions $f^+$ and $f^-$ are the \emph{positive variation} and the \emph{negative variation} of $f$, respectively. They could be also defined by taking into consideration only the positive or only the negative contributions in~\eqref{vitali}, respectively, instead of taking absolute values. However, this aspect is not important for our paper, so we do not pursue it any further.

\begin{proof}[Proof of Theorem~\ref{thbv} (a)]
Assume that $f$ is a right-continuous function on $[0,1]^d$ which has bounded HK$\mathbf{0}$-variation. Let $f^+,f^-$ be the functions in the Jordan decomposition of $f$ as in Theorem~\ref{thbv2}. As noted, both $f^+$ and $f^-$ are completely monotone.\\

Now we will show that $f^+$ and $f^-$ are right-continuous. We define functions $\tilde{f^+}$ and $\tilde{f^-}$ by setting
\begin{equation} \label{limits}
\tilde{f^+}(\mathbf{x}) = \lim_{\ve \searrow 0} f^+(x_1+\ve,x_2,\dots,x_d), \qquad \tilde{f^-}(\mathbf{x}) = \lim_{\ve \searrow 0} f^-(x_1+\ve,x_2,\dots,x_d)
\end{equation}
for $\mathbf{x} = (x_1,\dots,x_d) \in [0,1)^d$ and $\tilde{f^+}(\mathbf{x}) = f^+(\mathbf{x})$ and $\tilde{f^-}(\mathbf{x})=f^-(\mathbf{x})$ for $\mathbf{x} \in [0,1]^d \backslash [0,1)^d$. Note that the limits in~\eqref{limits} exist since $f^+$ and $f^-$ are monotone in every coordinate and bounded. By construction, the functions $\tilde{f^+}$ and $\tilde{f^-}$ are right-continuous in the first coordinate, at every point. Also, both functions $\tilde{f^+}$ and $\tilde{f^-}$ are completely monotone (this property is inherited from $f^+$ and $f^-$, respectively). Furthermore $\varhkn \tilde{f^+} = \tilde{f^+}(\mathbf{1}) - \tilde{f^+}(\mathbf{0}) \leq f^+(\mathbf{1}) - f^+(\mathbf{0})= \varhkn f^+$, and a similar inequality holds for $\tilde{f^-}$. We also have
\begin{eqnarray}
\tilde{f^+}(\mathbf{x}) - \tilde{f^-}(\mathbf{x}) & = & \lim_{\ve \searrow 0} \left( f^+(x_1+\ve,x_2,\dots,x_d) - f^-(x_1+\ve,x_2,\dots,x_d) \right) \nonumber\\
& = & \lim_{\ve \searrow 0} f(x_1+\ve,x_2,\dots,x_d) - f(\mathbf{0}) \label{rk1}\\
& = & f(\mathbf{x}) - f(\mathbf{0}),\label{rk2}
\end{eqnarray}
where we used the fact that $f$ is right-continuous to get from~\eqref{rk1} to~\eqref{rk2}. Thus by~\eqref{trian} we must actually have
\begin{equation} \label{varand}
\varhkn \tilde{f^+} = \varhkn f^+ \qquad \textrm{and} \qquad \varhkn \tilde{f^-} = \varhkn f^-,
\end{equation}
and
$$
\varhkn \tilde{f^+} + \varhkn \tilde{f^-} = \varhkn f.
$$
Since by construction $\tilde{f^+}(\mathbf{1})=f^+(\mathbf{1})$ and $\tilde{f^-}(\mathbf{1})=f^-(\mathbf{1})$, and since the functions $\tilde{f^+}$ and $\tilde{f^-}$ are completely monotone, by~\eqref{varand} we have
$$
\tilde{f^+}(\mathbf{0}) = \tilde{f^+} (\mathbf{1}) - \varhkn \tilde{f^+} = f^+ (\mathbf{1}) - \varhkn f^+ = f^+ (\mathbf{0}) = 0,
$$
and similarly $\tilde{f^-}(\mathbf{0})=0$. Overall, the functions $\tilde{f^+}$ and $\tilde{f^-}$ have all the properties requested from the decomposition in Theorem~\ref{thbv2}. However, since this decomposition is unique, we must have $\tilde{f^+}=f^+$ and $\tilde{f^-}=f^-$. In other words, the functions $f^+$ and $f^-$ must already be right-continuous in the first coordinate. The same argument can be applied to show that $f^+$ and $f^-$ must be right-continuous in all other coordinates as well.\\

We can define a set-function $\nu_{f}^+$ on the elements of $\mathcal{A}^*$ by setting
$$
\nu_{f}^+ ([\mathbf{0},\mathbf{x}]) = f^+(\mathbf{x}) \qquad \textrm{for $\mathbf{x} \in [0,1]^d$}.
$$
This function can be extended in a unique way into a countably additive set-function on the algebra of all finite unions of elements of $\mathcal{A}^*$ (here it is important that $f^+$ is right-continuous, to ensure that the resulting set-function is countably additive). Finally, by the Caratheodory extension theorem, this countably additive set-function can be extended in a unique way into a measure $\nu_f^+$ on the sigma-field generated by $\mathcal{A}^*$. Since the sigma-field generated by $\mathcal{A}^*$ consists of the Borel sets on $[0,1]^d$, the measure $\nu_f^+$ is a Borel measure. All the necessary extension theorems for this construction are contained in Chapter~5 of~\cite{yeh}; however, this is just the standard construction how a multivariate distribution function (namely the function $f^+$) defines a measure, as described for example in Chapter 3 of~\cite{kho}. In the same way, we can construct a measure $\nu_{f}^-$ from $f^-$. Note that both $\nu_{f}^+$ and $\nu_{f}^-$ are finite.\\

Let $\delta_{f(\mathbf{0})}$ be the signed Borel measure on $[0,1]^d$ for which
$$
\delta_{f(\mathbf{0})}(A) = \left\{ \begin{array}{ll} f(\mathbf{0}) & \textrm{if $\mathbf{0} \in A$}\\0 & \textrm{otherwise.}\end{array}\right.
$$
We define
\begin{equation} \label{decom}
\nu_f = \nu_{f}^+ - \nu_{f}^- + \delta_{f(\mathbf{0})}.
\end{equation}
Then $\nu_f$ is a finite signed Borel measure, and we have
$$
\nu_f([\mathbf{0},\mathbf{x}]) = f^+(\mathbf{x}) -f^-(\mathbf{x}) + f(\mathbf{0}) = f(\mathbf{x}) \qquad \textrm{for} \qquad \mathbf{x} \in [0,1]^d.
$$ 
By the Jordan Decomposition Theorem (see the end of the previous section) there exist measures $\nu^+$ and $\nu^-$ such that $\nu^+ \perp \nu^-$ and $\nu_f = \nu^+ - \nu^-$. Furthermore there exist two (Borel-)sets $C^+,C^-$ such that $[0,1]^d = C^+ \cup C^-$ and $C^+ \cap C^- = \emptyset$, and such that $\nu^-(C^+)=0$ and $\nu^+(C^-)=0$. It is not a priori clear that the Jordan decomposition of the measure $\nu_f$ is in direct correspondence with the Jordan decomposition of the function $f$, that $\nu_f^+ \perp \nu_f^-$, and that~\eqref{decom} already gives the Jordan representation of $\nu_f-\delta_{f(\mathbf{0})}$. However, we will now show that this actually is the case.\\

Let $(\nu^+,\nu^-)$ be the Jordan decomposition of $\nu_f$. We define functions $g^+(\mathbf{x})$ and $g^-(\mathbf{x})$ on $[0,1]^d$ by setting
$$
g^+(\mathbf{x}) = \nu^+([\mathbf{0},\mathbf{x}] \backslash \{\mathbf{0}\}), \qquad \mathbf{x} \in [0,1]^d,
$$
and
$$
g^-(\mathbf{x}) = \nu^-([\mathbf{0},\mathbf{x}] \backslash \{\mathbf{0}\})), \qquad \mathbf{x} \in [0,1]^d.
$$
Then it is easily seen that $g^+$ and $g^-$ are completely monotone functions, and that we have $f(\mathbf{x})=f(\mathbf{0}) + g^+ (\mathbf{x}) - g^-(\mathbf{x})$ for all $\mathbf{x} \in [0,1]^d$. Furthermore $g^+(\mathbf{0})=g^-(\mathbf{0})=0$. Let $C^+,C^-$ be the two sets from above. Then
\begin{eqnarray*}
\varhkn g^+ & = & g^+(\mathbf{1}) - \underbrace{g^+(\mathbf{0})}_{=0} \\
& = & \nu^+([\mathbf{0},\mathbf{1}] \backslash \{\mathbf{0}\})) \\
& = & \nu^+(C^+ \backslash \{\mathbf{0}\}) \\
& = & \nu_f(C^+ \backslash \{\mathbf{0}\}) \\
& = & \nu_f^+(C^+) - \nu_f^-(C^+) \\
& \leq & \nu_f^+([\mathbf{0},\mathbf{1}]) \\
& = & f^+(\mathbf{1}) \\
& = & \varhkn f^+.
\end{eqnarray*}
Similarly we obtain $\varhkn g^- \leq  \varhkn f^-$, which implies
\begin{eqnarray} 
\varhkn g^+ + \varhkn g^- & \leq & \varhkn f^+ + \varhkn f^- \label{equin} \\
& = & \varhkn f.\nonumber
\end{eqnarray}
By~\eqref{trian} the inequality sign in~\eqref{equin} must actually be an equality sign. Consequently, the decomposition $f(\mathbf{x})=f(\mathbf{0}) + g^+ (\mathbf{x}) - g^-(\mathbf{x})$ is a decomposition having all the properties described in Theorem~\ref{thbv2}. Now since the decomposition in Theorem~\ref{thbv2} is unique, this implies that $f^+=g^+$ and $f^-=g^-$. As a consequence, we have 
$$
\vart \nu_f = \nu^+\left([0,1]^d\right) + \nu^-\left([0,1]^d\right) = f^+(\mathbf{1}) + f^-(\mathbf{1}) + |f(\mathbf{0})| = \varhkn f + |f(\mathbf{0})|.
$$
This proves part (a) of the theorem.
\end{proof}

\begin{proof}[Proof of Theorem~\ref{thbv} (b)]
Assume that a finite signed measure $\nu$ on $[0,1]^d$ is given. Then by the Jordan decomposition theorem there exist two finite measures $\nu^+$ and $\nu^-$ such that $\nu=\nu^+ - \nu^-$, and such that $\nu^+ \perp \nu^-$. We define two functions $f^+,f^-$ on $[0,1]^d$ by setting
$$
f^+(\mathbf{x}) = \nu^+\left([\mathbf{0},\mathbf{x}] \backslash \{\mathbf{0}\}\right), \qquad f^-(\mathbf{x}) = \nu^-\left([\mathbf{0},\mathbf{x}] \backslash \{\mathbf{0}\}\right), \qquad \mathbf{x} \in [0,1]^d.
$$
Then $f^+$ and $f^-$ are two finite, right-continuous, completely monotone functions. Furthermore, the function $f(\mathbf{x})=f^+(\mathbf{x})-f^-(\mathbf{0})+\nu(\{\mathbf{0}\}),~\mathbf{x} \in [0,1]^d,$ is a right-continuous function of bounded HK$\mathbf{0}$-variation (since it can be written as the difference of two finite, completely monotone functions; see~\cite[Corollary 3]{leon}). As explained in part (a) of this proof, the function $f$ defines measures $\nu_f^+$ and $\nu_f^-$ and a finite signed measure $\nu_f$. It is then easily seen that $\nu_f$ coincides with $\nu$, that the pair $(\nu_f^+,\nu_f^-)$ is the unique Jordan decomposition of $\nu_f - \delta_{f(\mathbf{0})}$, and that consequently $f^+$ and $f^-$ are the Jordan decomposition of the function $f$. As a consequence we have
\begin{eqnarray*}
\varhkn f & = & \varhkn f^+ + \varhkn f^- \\
& = & \nu^+\left([0,1]^d \backslash\{\mathbf{0}\}\right) + \nu^-\left([0,1]^d \backslash\{\mathbf{0}\}\right) \\
& = & \vart \nu - \underbrace{|\nu(\{\mathbf{0}\})|}_{=|f(\mathbf{0})|},
\end{eqnarray*}
which proves the theorem.
\end{proof}

\begin{proof}[Proof of Lemma~\ref{lemmahk0}]
Let a function $f$ on $[0,1]^d$ be given, and assume that $f$ has bounded HK$\mathbf{0}$-variation. In the first step we assume that $f$ is right-continuous. Then by Theorem~\ref{thbv2} and Theorem~\ref{thbv} there exist completely monotone functions $f^+$ and $f^-$ such that~\eqref{sumofv} holds, and there exists a signed measure $\nu$ which, together with its Jordan decomposition $\nu = \nu^+ - \nu^-$, satisfies~\eqref{ths1}-\eqref{ths3}.\\

To calculate the HK-variation of $f$, we have to calculate its variation in the sense of Vitali on faces of the form~\eqref{face1}. However, the situation becomes much easier if we separately calculate the variations in the sense of Vitali of $f^+$ and $f^-$ instead. For a completely monotone function $h$ we have
\begin{equation} \label{vdequ}
V^{(d)}\left(h;[0,1]^d\right) = \Delta^{(d)}\left(h;[0,1]^d\right).
\end{equation}
This equality follows from the fact that for a completely monotone function all summands in the sum in~\eqref{vitali} are non-negative, and consequently this sum is a sort of telescoping sum, no matter which partition $\mathcal{P}$ is chosen  (alternatively, this equality may be deduced from the fact that it trivially holds in the one-dimensional case, and that the $d$-dimensional difference operator $\Delta^{(d)}$ actually is the composition of $d$ one-dimensional difference operators). For the same reason, the same equality as~\eqref{vdequ} holds for the $s$-dimensional variation in the sense of Vitali $V^{(s)}$ on every $s$-dimensional face of $[0,1]^d$, for $1 \leq s \leq d$. In particular we have
\begin{equation} \label{vdequ2}
V^{(s)}\left(h;U_d^{(i_1,\dots,i_s)}\right) = \Delta^{(s)}\left(h;U_d^{(i_1,\dots,i_s)}\right),
\end{equation}
provided $h$ is completely monotone.\\

A little combinatorial reasoning now shows that the representation~\eqref{ths3} implies that for any $s$-dimensional face $U_d^{(i_1,\dots,i_s)}$ of the form~\eqref{face1}, $1 \leq s \leq d$, we have
$$
\Delta^{(s)}\left(f^+;U_d^{(i_1,\dots,i_s)}\right) = \nu^+ \left(\left\{ (x_1,\dots,x_d) \in [0,1]^d:~ x_{i_1}>0,x_{i_2}>0,\dots,x_{i_s}>0 \right\} \right)
$$
(this is how a distribution function is used to calculate the measure of a half-open axis-parallel box). Together with~\eqref{vdequ2} this implies
\begin{eqnarray}
V^{(s)}\left(f^+;U_d^{(i_1,\dots,i_s)}\right) & \leq & \nu^+ \left(\left\{ (x_1,\dots,x_d) \in [0,1]^d:~ x_{i_1}>0,x_{i_2}>0,\dots,x_{i_s}>0 \right\} \right) \nonumber\\
& \leq & \nu^+([0,1]^d \backslash \{\mathbf{0}\}) ~=~ f^+(\mathbf{1}) ~=~ \varhkn f^+.  \label{nubou}
\end{eqnarray}
We note that the number of summands in the definition of the HK-variation is $2^d-1$. Since by~\eqref{nubou} each of these summands is bounded by $\nu^+([0,1]^d)$, we have
$$
\varhk f^+ \leq \left(2^d -1\right) \varhkn f^+.
$$
In a similar way we obtain
$$
\varhk f^- \leq \left(2^d -1\right) \varhkn f^-.
$$
Using~\eqref{sumofv} and~\eqref{trian2} we obtain
$$
\varhk f \leq \varhk f^+ + \varhk f^- \leq \left(2^d-1\right) \varhkn f,
$$
which proves the lemma under the additional assumption that $f$ is right-continuous.\\

Without assuming that $f$ is right-continuous we still can use Theorem~\ref{thbv2} to find a Jordan decomposition $f=f^+-f^-$ of $f$, but the functions $f^+$ and $f^-$ are (in general) not right-continuous and consequently cannot be used to define a measure (as in Theorem~\ref{thbv}). However, as~\eqref{vdequ} and~\eqref{vdequ2} show, the variation in the sense of Vitali (on $[0,1]^d$ as well as on lower-dimensional faces) of a completely monotone function $h$ depends only on the values of $h$ at the \emph{corners} of $[0,1]^d$; and consequently the same must be true for the HK$\mathbf{0}$-variation and the HK-variation of $h$. We set
$$
\overline{f^+}(x_1,\dots,x_d) = f(\tau(x_1),\dots,\tau(x_d)), \qquad \textrm{for $\mathbf{x}=(x_1,\dots,x_d) \in [0,1]^d$},
$$
where $\tau(y) = 0$ for $0\leq y < 1$ and $\tau(1)=1$. Informally speaking, the value of $\overline{f^+}$ at a point $\mathbf{x}$ is the value of $f$ at the maximal corner of $[0,1]^d$ which is $\leq \mathbf{x}$. Trivially $\overline{f^+}$ coincides with $f^+$ on all corners of $[0,1]^d$, and it is easy to see that $\overline{f^+}$ is also completely monotone. Furthermore, $\overline{f^+}(\mathbf{0})=0$ and $\overline{f^+}$ is right-continuous. Thus, since we have already shown the lemma for right-continuous functions, we have
$$
\underbrace{\varhk \overline{f^+}}_{= \varhk f^+} \leq \left(2^d-1\right) \underbrace{\varhkn \overline{f^+}}_{=\varhkn f^+}.
$$
A similar inequality holds for $f^-$. Together with~\eqref{sumofv} and~\eqref{trian2} this proves the lemma also in the case when $f$ is not right-continuous.\\

To show that the statement of the lemma also holds when the role of the HK$\mathbf{0}$-variation and the HK-variation are interchanged, we define $g(\mathbf{x})=f(\mathbf{1}-\mathbf{x})$ for $\mathbf{x} \in [0,1]^d$. Then by~\eqref{fgrel} and by the version of the lemma which we proved above we have
$$
\varhkn f = \varhk g \leq \left(2^d-1\right) \varhkn g = \left(2^d-1\right) \varhk f.
$$
\end{proof}

\section{A Koksma--Hlawka inequality for general measures:~a proof of Theorem~\ref{th1}} \label{seckh}

The proof of Theorem~\ref{th1} follows similar proofs in~\cite[Theorem C.1.4]{bl} and~\cite[Theorem 3.2]{okter}.\\

Let $\mathbf{x}_1, \dots, \mathbf{x}_N$ be given. Throughout the proof we may assume without loss of generality that $f(\mathbf{1})=0$ (since otherwise we may replace $f(\mathbf{x})$ by $f(\mathbf{x})-f(\mathbf{1})$, which changes neither the left-hand side nor the right-hand side of~\eqref{th1equ}).\\

In a first step, we assume that $f$ is left-continuous. We define the function $g(\mathbf{x})=f(\mathbf{1}-\mathbf{x})$ for $\mathbf{x} \in [0,1]^d$. Since we assumed that $f$ is left-continuous and $f(\mathbf{1})=0$, this clearly implies that $g$ is right-continuous and $g(\mathbf{0})=0$. Furthermore, by~\eqref{fgrel}, we have $\varhkn g = \varhk f$. Now we apply Theorem~\ref{thbv} to the function $g$. Let $\nu$ be the signed measure from Theorem~\ref{thbv} which is defined by $g$. Let $\hat{\nu}$ be the ``reflected'' measure of $\nu$, which satisfies
$$
\hat{\nu} (A) = \nu(\mathbf{1}-A),
$$
for any Borel set $A \subset [0,1]^d$, where we write $\mathbf{1}-A = \left\{\mathbf{1} - \mathbf{x}:~\mathbf{x} \in A \right\}$. It is easily verified that the fact that $\nu$ is a signed Borel measure implies that $\hat{\nu}$ is a signed Borel measure as well, and that they have the same total variation. Let $|\hat{\nu}|$ be the variation measure of $\hat{\nu}$ (see the end of Section~\ref{secdef}); then according to the previous remarks and Theorem~\ref{thbv} we have
\begin{eqnarray} 
\vart \hat{\nu} & = & \varhkn g + |g(\mathbf{0})| \nonumber\\
& = & \varhk f.\label{varhat}
\end{eqnarray}
Now on the one hand we have
\begin{eqnarray*}
\frac{1}{N}\sum_{n=1}^N f(\mathbf{x}_k) & = & \frac{1}{N} \sum_{n=1}^N g(\mathbf{1}-\mathbf{x}_k) \\
& = & \int_{[0,1]^d} \frac{1}{N} \sum_{n=1}^N \mathds{1}_{[\mathbf{0},\mathbf{1}-\mathbf{x}_n]} (\mathbf{y}) d\nu(\mathbf{y}) \\
& = & \int_{[0,1]^d} \frac{1}{N} \sum_{n=1}^N \mathds{1}_{[\mathbf{x_n},\mathbf{1}]} (\mathbf{y}) d\hat{\nu}(\mathbf{y}) \\
& = &  \int_{[0,1]^d} \frac{1}{N} \sum_{n=1}^N \mathds{1}_{[\mathbf{0},\mathbf{y}]} (\mathbf{x}_n) d\hat{\nu}(\mathbf{y}). 
\end{eqnarray*}
On the other hand, in a similar way, by Fubini's theorem we have
\begin{eqnarray*}
\int_{[0,1]^d} f(\mathbf{x}) d\mu(\mathbf{x}) & = & \int_{[0,1]^d} \int_{[0,1]^d} \mathds{1}_{[\mathbf{0},\mathbf{1}-\mathbf{x}]}(\mathbf{y}) d\nu(\mathbf{y}) d\mu(\mathbf{x}) \\
& = & \int_{[0,1]^d} \int_{[0,1]^d} \mathds{1}_{[\mathbf{x},\mathbf{1}]}(\mathbf{y}) d\hat{\nu}(\mathbf{y}) d\mu(\mathbf{x}) \\
& = & \int_{[0,1]^d} \int_{[0,1]^d} \mathds{1}_{[\mathbf{0},\mathbf{y}]}(\mathbf{x}) d\hat{\nu}(\mathbf{y}) d\mu(\mathbf{x})\\
& = & \int_{[0,1]^d} \int_{[0,1]^d} \mathds{1}_{[\mathbf{0},\mathbf{y}]}(\mathbf{x}) d\mu(\mathbf{x}) d\hat{\nu}(\mathbf{y}) \\
& = & \int_{[0,1]^d} \mu([\mathbf{0},\mathbf{y}]) d\hat{\nu}(\mathbf{y}).
\end{eqnarray*}
Consequently
\begin{eqnarray*}
\left| \frac{1}{N}\sum_{n=1}^N f(\mathbf{x}_k)  - \int_{[0,1]^d} f(\mathbf{x}) d\mu(\mathbf{x}) \right| & \leq & \int_{[0,1]^d} \underbrace{\left| \frac{1}{N} \sum_{n=1}^N \mathds{1}_{[\mathbf{0},\mathbf{y}]} (\mathbf{x}_n) - \mu([\mathbf{0},\mathbf{y}]) \right|}_{\leq D_N^*(\mathbf{x}_1, \dots, \mathbf{x}_N;\mu)} d|\hat{\nu}|(\mathbf{y}) \\
& \leq & D_N^*(\mathbf{x}_1, \dots, \mathbf{x}_N;\mu) ~\vart \hat{\nu}.
\end{eqnarray*}
Together with~\eqref{varhat} this proves the Theorem in the case that $f$ is left-continuous.\\

Now we show that we can reduce the general case to the case of $f$ being left-continuous. Let $f$ be given. By the strong law of large numbers and by the multidimensional Glivenko--Cantelli theorem (see for example~\cite[Chapter 26]{sw}) for any $\ve>0$ there exist a number $M$ and points $\mathbf{y}_1, \dots, \mathbf{y}_M \in [0,1]^d$ such that the two inequalities
\begin{equation} \label{deps2}
\left| \frac{1}{M} \sum_{m=1}^M f(\mathbf{y}_m) - \int_{[0,1]^d} f(\mathbf{x})~d\mu \right| \leq \ve
\end{equation}
and
\begin{equation} \label{deps}
D_N^*(\mathbf{y}_1, \dots, \mathbf{y}_M;\mu) \leq \ve
\end{equation}
are both satisfied. Set
$$
\mathcal{G} = \{\mathbf{0},\mathbf{1},\mathbf{x}_1,\dots,\mathbf{x}_N,\mathbf{y}_1,\dots,\mathbf{y}_M\},
$$
and let $\mathcal{H}$ be the $d$-dimensional grid that is generated by the elements of $\mathcal{G}$; that is, the set of all points in $\mathbf{z} \in [0,1]^d$ such that the $j$-th coordinate of $\mathbf{z}$ appears as the $j$-th coordinate of an element of $\mathcal{G}$, for $1 \leq j \leq d$. For $\mathbf{x} \in [0,1]^d$, let $\textup{succ}(\mathbf{x})$ denote the uniquely defined element $\mathbf{z}$ of $\mathcal{H}$ for which $\mathbf{x} \leq \mathbf{z}$ and for which $\mathbf{z} \leq \mathbf{y}$ for all $\mathbf{y} \in \mathcal{H}:~\mathbf{x} \leq \mathbf{y}$. Informally speaking, $\textup{succ}(\mathbf{x})$ is the smallest element of $\mathcal{H}$ which is $\geq \mathbf{x}$ (that is, $\textup{succ}(\mathbf{x})$ is the \emph{successor} of $\mathbf{x}$ within $\mathcal{H}$). We define a function $\tilde{f}$ by setting
$$
\tilde{f}(\mathbf{x}) = f(\textup{succ}(\mathbf{x})), \qquad \textbf{x} \in [0,1]^d.
$$
Note that by construction $\tilde{f}(\mathbf{z})=f(\mathbf{z})$ for all points $z \in \mathcal{G}$. Furthermore, by construction the function $\tilde{f}$ is left-continuous. Additionally, it is easily seen that $\varhk \tilde{f} \leq \varhk f$. Since we have already proved the theorem for left-continuous functions, we get
\begin{eqnarray}
& & \left| \frac{1}{N} \sum_{n=1}^N \underbrace{f(\mathbf{x}_n)}_{=\tilde{f}(\mathbf{x}_n)} - \int_{[0,1]^d} f(\mathbf{x})~d\mu \right| \nonumber\\
& \leq & \left| \frac{1}{N} \sum_{n=1}^N \tilde{f}(\mathbf{x}_n) - \int_{[0,1]^d} \tilde{f}(\mathbf{x})~d\mu \right| + \left|\int_{[0,1]^d} \tilde{f}(\mathbf{x})~d\mu - \frac{1}{M} \sum_{m=1}^M \underbrace{\tilde{f}(\mathbf{y}_m)}_{=f(\mathbf{y}_m)} \right| \label{term1}\\
& & + \left| \frac{1}{M} \sum_{m=1}^M f(\mathbf{y}_m) - \int_{[0,1]^d} f(\mathbf{x})~d\mu \right|. \label{term2}
\end{eqnarray}
The first term in~\eqref{term1} is at most $(\varhk \tilde{f}) D_N^*(\mathbf{x}_1,\dots,\mathbf{x}_N;\mu)$, since $\tilde{f}$ is left-continuous. The second term in~\eqref{term1} is at most $\ve \varhk \tilde{f}$, also since $\tilde{f}$ is left-continuous and by~\eqref{deps}. Finally, the term in~\eqref{term2} is at most $\ve$ by~\eqref{deps2}. Since $\varhk \tilde{f} \leq \varhk f$ and since $\ve>0$ was arbitrary, this proves the Theorem.

\section{Transformations of point sets and Chelson's general Koksma--Hlawka inequality} \label{sectrans}

In this section we will present the transformation method proposed by Chelson~\cite{chelson}, which supposedly transforms a low-discrepancy point set with respect to the uniform measure into a low discrepancy point set with respect to a general measure $\mu$. We will show, contrary to what is claimed in~\cite{chelson}, that this transformation method generally fails, and only gives the desired result in the case when $\mu$ is of product type.\\

Before turning to Chelson's method, we want to note that the problem of transforming a low-discrepancy sequence with respect to the uniform measure into a low-discrepancy sequence with respect to another measure $\mu$ has been considered by several other authors, for example in~\cite{hk,hm1,hm2,sn}. Let $\mathbf{x}_1,\dots,\mathbf{x}_N$ be a point set in $[0,1]^d$. If $\mu$ is of product type, that is if it is the $d$-dimensional product measure of $d$ one-dimensional measures, and if $\mu$ has a density (with respect to $\lambda$), then in~\cite{hk,hm2} transformation methods are presented which (in a computationally tractable way) generate a sequence $\mathbf{y}_1,\dots,\mathbf{y}_N \in [0,1]^d$ such that
$$
D_N^*(\mathbf{y}_1,\dots,\mathbf{y}_N;\mu) \leq c(d,\mu) D_N^*(\mathbf{x}_1,\dots,\mathbf{x}_N).
$$
In the case of more general measures, the known results are much less satisfactory. Even under some technical assumptions on $\mu$, the best known transformation~\cite{hm2} only gives
$$
D_N^*(\mathbf{y}_1,\dots,\mathbf{y}_N;\mu) \leq c(d,\mu) \left(D_N^*(\mathbf{x}_1,\dots,\mathbf{x}_N)\right)^{1/d}.
$$
Chelson claims that even in the general case one can reach
$$
D_N^*(\mathbf{y}_1,\dots,\mathbf{y}_N;\mu) = D_N^*(\mathbf{x}_1,\dots,\mathbf{x}_N),
$$
but as noted his ``proof'' of this assertion is incorrect and must be dismissed.\\

Now we turn to the description of the transformation method suggested in~\cite{chelson}. We change the notation (in such a way that it fits together with the rest of this paper) and simplify some statements, but our exposition is a truthful re-narration of the presentation in~\cite{chelson}. Let $g(y_1,\dots,y_d)$ be a probability density on $[0,1]^d$, let $G(y_1,\dots,y_d)$ be its distribution function, and let $\mu_G$ be the corresponding probability measure. We require that $g$ is non-zero on $[0,1]^d$. Let $g_1$ be defined by
$$
g_1(y_1) = \underbrace{\int_0^1 \dots \int_0^1}_{d-1 \textrm{~integrals}} g(y_1,\dots,y_d)~dy_2 \dots dy_d;
$$
that is, $g_1$ is the marginal density for $y_1$. Let $G_1(y_1)$ be the distribution function of $g_1$, and let $G_1^{-1}$ be its inverse (which exists because $g$ is positive). Similarly, for $2 \leq s < d$, let
$$
g_{1,\dots,s} (y_1,\dots,y_s) = \underbrace{\int_0^1 \dots \int_0^1}_{d-s \textrm{~integrals}} g(y_1,\dots,y_s) dy_{s+1}\dots dy_{d}
$$
be the marginal density for $y_1, \dots, y_s$. Note that $g_{1,\dots,d}=g$. Then, for $1 \leq s \leq d$, let
$$
g_s (y_s) = g(y_s|y_1,\dots,y_{s-1}) = \frac{g_{1,\dots,s}(y_1,\dots,y_s)}{g_{1,\dots,s-1}(y_1,\dots,y_{s-1})}.
$$
Furthermore, let $G_s(y_s)$ be the (conditional) distribution function of $g_s(y_s)$, that is
\begin{equation} \label{gsdef}
G_s(y_s) = G_s(y_s|y_1,\dots,y_{s-1}) = \int_0^{y_s} g(u_s|y_1,\dots,y_{s-1})~du_s.
\end{equation}
Let $G_s^{-1}$ denote the inverse of $G_s$.\\

Now Chelson introduced the transformation as follows:
\begin{quote}
Let $\mathbf{x}=(x_1,\dots,x_d)\in [0,1]^d$ be given. First, set $z_1=G_1^{-1}(x_1)$. Then set $z_2=G_2^{-1}(x_2)$, and so on, until $z_d=G_d^{-1}(x_d)$. This gives a number $\mathbf{z}=(z_1,\dots,z_d)\in[0,1]^d$. Note that the values $z_1,z_2,\dots$ must be calculated sequentially, since each depends on those which are already chosen.\footnote{The dependent nature of the inverse functions is suppressed in Chelson's notation. What is meant is that $z_2=G_2^{-1}(x_2|z_1)$, \dots, $z_d=G_d^{-1}(x_d|z_1,\dots,z_{d-1})$, which explains why the values $z_1,\dots,z_d$ have to be calculated consecutively.} We write $T$ for the transformation which maps $\mathbf{x} \mapsto T \mathbf{x} = \mathbf{z}$ in this way.
\end{quote}
It is not mentioned in~\cite{chelson}, but this transformation is the \emph{Rosenblatt transformation}, which was introduced in~\cite{rosen}.\\

Chelson proves that if $X$ is a uniformly distributed random variable on $[0,1]^d$, then $Z=TX$ has distribution $\mu_G$. This is true, and was also shown in~\cite{rosen}. However, Chelson also claims the following (\cite[Theorem 2-5]{chelson}):
\begin{quote}
Let $\mathbf{x}_1,\dots,\mathbf{x}_N$ be a sequence in $[0,1]^d$, and let $\mathbf{z}_1=T\mathbf{x}_1,\dots,\mathbf{z}_N=T \mathbf{x}_N,$ be its image under the transformation $T$ described above. Then
\begin{equation} \label{discrch}
D_N^*(\mathbf{z}_1,\dots,\mathbf{z}_N;\mu_G)=D_N^*(\mathbf{x}_1,\dots,\mathbf{x}_N).
\end{equation}
\end{quote}
For the ``proof'', Chelson~\cite[p. 29]{chelson} argues as follows:
\begin{quote}
We define a vector-valued function $\tilde{G}(\mathbf{z}):~[0,1]^d \mapsto [0,1]^d$ by
\begin{equation} \label{tildeg}
\tilde{G}(\mathbf{z}) = (G_1(z_1),\dots,G_d(z_d)).
\end{equation}
Then, by construction, for any $\mathbf{a} \in [0,1]^d$
\begin{equation} \label{tildeg2}
\sum_{n=1}^N \mathds{1}_{[\mathbf{0},\mathbf{a}]} (\mathbf{z}_n) = \sum_{n=1}^N \mathds{1}_{\left[\mathbf{0},\tilde{G}(\mathbf{a})\right]} (\mathbf{x}_n),
\end{equation}
since by $\mathbf{z} \leq \mathbf{a}$ for $\mathbf{z} = T \mathbf{x}$ we mean
\begin{eqnarray*}
z_1 & = & G_1^{-1}(x_1) \leq a_1,\\
& \vdots & \\
z_d & = & G_d^{-1} (x_d) \leq a_d,
\end{eqnarray*}
and this occurs if and only if $x_s \leq G_s(a_s)$ for $1 \leq s \leq d$.
\end{quote}
This looks reasonable, but it is actually false. We will give a counterexample below. Chelson continues claiming that
\begin{quote}
\begin{equation} \label{tildeg3}
\mu_G([\mathbf{0},\mathbf{a}]) = \lambda\left(\left[\mathbf{0},\tilde{G}(\mathbf{a})\right]\right).
\end{equation}
\end{quote}
This is also false (see also the counterexample below). The error comes from treating the dependent functions 
$$
G_1(y_1),G_2(y_2|y_1),\dots,G_d(y_d|y_1,\dots,y_{d-1})
$$ 
as independent functions 
$$
G_1(y_1),G_2(y_2),\dots,G_d(y_d).
$$
In the definition~\eqref{gsdef} it is suggested that the dependent nature of $G_1,\dots,G_d$ is suppressed merely in order to shorten formulas; however, it seems that by suppressing the dependent nature of $G_1,\dots,G_d$ in the notation, Chelson made the error of simply ignoring this dependence, which of course is not justified.\footnote{To be more precise, the dependent nature of $G_1,\dots,G_d$ is only then irrelevant when $G$ is of product form; in this case the conditional one-dimensional distributions are just the (unconditional) one-dimensional distributions of the product representation themselves, and the transformation method works correctly; see Theorem~\ref{thinv} below.} A correct form of~\eqref{tildeg2} and~\eqref{tildeg3} would require the set $\{\tilde{G}(\mathbf{y}):~\mathbf{y} \in [\mathbf{0},\mathbf{a}]\}$ instead of $[\mathbf{0},\tilde{G}(\mathbf{a})]$ on the right-hand side of the equation; however, because of the dependent nature of $G_1,\dots,G_d$ the set $\{\
\tilde{G}(\mathbf{y}):~\mathbf{y} \in [\mathbf{0},\mathbf{a}]\}$ is in general \emph{not} an axis-parallel box, and consequently the discrepancy of $\mathbf{x}_1,\dots,\mathbf{x}_N$ cannot be used to estimate the number of elements of the original point set which are contained in such a set (see the example below).\\

In dimension $d=1$, Chelson's claim is right (there are no conditional distributions in this case). The simplest counterexample can be given for $d=2$ and $N=1$. For example, let the density $g$ be given by 
$$
g(\mathbf{y}) = \left\{ \begin{array}{ll} 1/2 & \textrm{if $y_1 \leq y_2$} \\ 3/2 & \textrm{if $y_1 > y_2,$} \end{array} \right. \qquad \textrm{for~} \mathbf{y}=(y_1,y_2) \in [0,1]^2.
$$
Thus $g=1/2$ on the upper left half of the unit square divided by the diagonal linking $\mathbf{0}$ to~$\mathbf{1}$, and $g=3/2$ on the lower right half. Clearly $g$ is a probability density, and $g$ is positive. Note that $g$ cannot be written as the product of two one-dimensional densities. Constructing $g_1$ and $g_2$ in the way described above, we get
\begin{eqnarray*}
g_1(y_1) & = & \int_0^1 g(y_1,y_2)~dy_2 = (3/2)y_1 + (1/2)(1-y_1) = y_1 + \frac{1}{2}\\
g_2(y_2|y_1) & = & \frac{g(y_1,y_2)}{g_1(y_1)} = \left\{ \begin{array}{ll} \frac{1}{1+2y_1} & \textrm{if $y_1 \leq y_2$} \\ \frac{3}{1+2y_1} & \textrm{if $y_1 > y_2.$} \end{array} \right.
\end{eqnarray*}
Accordingly, we get
\begin{eqnarray*}
G_1(y_1) & = & \frac{y_1^2+y_1}{2},\\
G_2(y_2|y_1) & = & \left\{ \begin{array}{ll} \frac{y_2+2y_1}{1+2y_1} & \textrm{if $y_1 \leq y_2$} \\ \frac{3y_2}{1+2y_1}  & \textrm{if $y_1 > y_2$} \end{array} \right.
\end{eqnarray*}

Let $\mathbf{x}$ be the point $\mathbf{x}=(x_1,x_2)=(56/81,20/23)$. The transformed point is $\mathbf{z} =T \mathbf{x} = (7/9,20/27)$, since 
$$
G_1\left(\frac{7}{9}\right)=\frac{56}{81} \qquad \textrm{and} \qquad G_2\left(\frac{20}{27}\Big|\frac{7}{9}\right)=\frac{20}{23}.
$$
Let $\mathbf{a}=(1,8/10)$. For the function $\tilde{G}$ defined in~\eqref{tildeg} we have $\tilde{G}(1,8/10)=(1,8/10)$. According to~\eqref{tildeg2} we should have
$$
\mathds{1}_{[\mathbf{0},\mathbf{a}]} (\mathbf{z}) = \mathds{1}_{\left[\mathbf{0},\tilde{G}(\mathbf{a})\right]} (\mathbf{x}).
$$
However, instead we have 
$$
\mathds{1}_{[\mathbf{0},\mathbf{a}]} (\mathbf{z}) = \mathds{1}_{\left[\mathbf{0},\left(1,\frac{8}{10}\right)\right]} \left(\left(\frac{7}{9},\frac{20}{27}\right)\right) = 1,
$$ 
but 
$$
\mathds{1}_{\left[\mathbf{0},\tilde{G}(\mathbf{a})\right]} (\mathbf{x})=\mathds{1}_{\left[\mathbf{0},\left(1,\frac{8}{10}\right)\right]} \left(\left(\frac{56}{81},\frac{20}{23}\right)\right)=0.
$$
Thus~\eqref{tildeg} fails to hold. In a similar way we can show that the claim~\eqref{tildeg3} is also false. We have
$$
\mu_G([\mathbf{0},\mathbf{a}]) = \int_{\left[\mathbf{0},\left(1,\frac{8}{10}\right)\right]} g(\mathbf{y})d\mu_G = \frac{22}{25} = 0.88,  \qquad \textrm{but} \qquad \lambda([\mathbf{0},\tilde{G}(\mathbf{a})]) = \frac{8}{10}.
$$
Now it is not surprising anymore that~\eqref{discrch} is also false. After some calculations we get
$$
D_N^*(\mathbf{z};\mu_G)=\mu_G\left(\left[\mathbf{0},\left(1,\frac{20}{27}\right)\right]\right) = \frac{610}{729} \approx 0.84,
$$
but
$$
D_N^*(\mathbf{x}) = \lambda\left(\left[\mathbf{0},\left(1,\frac{20}{23}\right)\right]\right)=\frac{20}{23} \approx 0.87.
$$
As already mentioned, the error is caused by ignoring the fact that the function $G_2(y_2)=G_2(y_2|y_1)$ depends on $y_1$ (and not only on $y_2$), and by the incorrect assumption that the image of an axis-parallel box under $\tilde{G}$ is again an axis-parallel box. The following picture shows, for the example given above, the set $A = \left\{\tilde{G}(\mathbf{y}):~\mathbf{y} \in [\mathbf{0},\mathbf{a}]\right\}$, the set $B=[\mathbf{0},\tilde{G}(\mathbf{a})]$, and the point $\mathbf{x}$. Note that the sets $A$ and $B$ do not coincide, and that $A$ is not an axis-parallel box.\\

\begin{figure}[H]
\begin{center}
\includegraphics[angle=0,width=80mm]{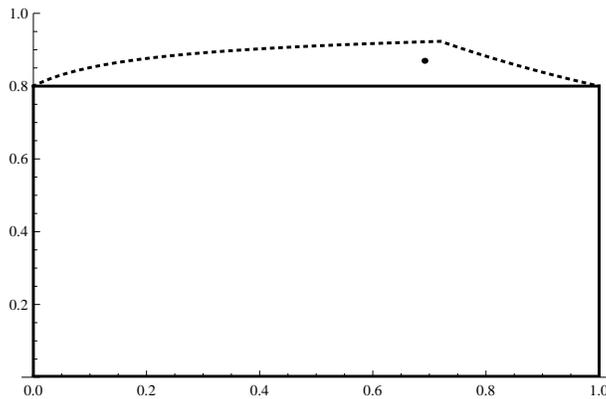}\\
\caption{The sets $A$ (everything below the dotted line) and $B$ (everything below the solid line), and the point $\mathbf{x}$. The point $\mathbf{x}$ is contained in $A$, which implies that $\mathbf{z}$ is contained in $[\mathbf{0},\mathbf{a}]$. However, $\mathbf{x}$ is not contained in $B$.}
\end{center}
\end{figure}

Chelson's generalized Koksma--Hlawka inequality is mentioned several times in the literature; it is explicitly stated, in differing forms, in~\cite{hk,hkt,okter,ra,rg,rna,sl,sm}. In some of these instances it is stated without the transformation procedure and without using the incorrect identity~\eqref{discrch} (which means that in those cases it is stated roughly in the same form as our Theorem~\ref{th1} or Corollary~\ref{th3}). In other cases, the transformation from uniform distribution to an other distribution and the (incorrect) identity~\eqref{discrch} are explicitly accentuated. For example, after stating Chelson's theorem in its original form, in~\cite{sm} the authors write 
\begin{quote}
It is important to note in [the statement of Chelson's theorem] that even though the sampling technique has changed the sequence used to evaluate [the function], the discrepancy appearing in the error bound is still that of the original sequence.
\end{quote}
A generalization of Chelson's results is apparently contained in the PhD thesis of Maize~\cite{maize}, which we cannot access (but judging from the statement of~\cite[Theorem 4.2]{sm}, it probably suffers from the same problem as Chelson's original results). In~\cite{okter} a generalization of Chelson's results is ``proved'', repeating  Chelson's errors. Actually, in~\cite{okter} the author proves the analogue of~\eqref{discrch} only in the one-dimensional case (in which it is true) and writes that \emph{the generalization to higher dimensions is straightforward} (which is actually not the case). It seems that the first time that a correct version of~\eqref{discrch} is stated in the literature is in~\cite[Theorem 5]{rg}, where it is assumed that the measure $\mu$ has a density which is the product of $d$ one-dimensional densities. Strangely, in~\cite{rg} this result is attributed to~\cite{okter}, where the additional assumption that $\mu$ is of product form does not appear (and the result is stated in 
an incorrect form, as noted above). However, it is easy to see that the assumption that $\mu$ possesses a density is not necessary in order to assure that the $\mu$-discrepancy of the transformed point set is dominated by the discrepancy of the original point set when $\mu$ is of product form. Thus a general correct version of~\eqref{discrch} would look as follows.

\begin{theorem} \label{thinv}
Let $\mu$ be a normalized Borel measure on $[0,1]^d$, and let $G(\mathbf{x})$ be its distribution function. Assume that there exist $d$ one-dimensional normalized Borel measures $\mu_1,\dots,\mu_d$ having distribution functions $G_1,\dots,G_d$ such that 
$$
G(x_1,\dots,x_d) = \prod_{s=1}^d G_s(x_s).
$$
Let 
$$
G_s^{-1}(y) = \min \{y \in [0,1]:~G_s(x) \geq y \}, \qquad y \in [0,1],
$$
be the (pseudo-)inverse of $G_s$ for $s=1,\dots,d$. Let $T$ be the transformation on $[0,1]^d$ which maps a point $\mathbf{x} \in [0,1]^d$ to
$$
T(\mathbf{x}) = T(x_1,\dots,x_d) = \left( G_1^{-1}(x_1),\dots,G_d^{-1}(x_d)\right).
$$
Let $\mathbf{x}_1, \dots, \mathbf{x}_N$ be a set of points in $[0,1]^d$. Then
\begin{equation} \label{thconcl}
D_N^* \left(T(\mathbf{x}_1),\dots,T(\mathbf{x}_N);\mu \right) \leq D_N^*(\mathbf{x}_1,\dots,\mathbf{x}_N).
\end{equation}
If all the functions $G_1,\dots,G_d$ are invertible, then in~\eqref{thconcl} we have equality.
\end{theorem}

\begin{proof}
We define a function $\tilde{T}$ which maps $\mathbf{y}=(y_1,\dots,y_d)$ to $\tilde{T}(\mathbf{y}) = (G_1(y_1),\dots,G_d(y_d))$. Let $A \in \mathcal{A}^*$ be given, and let $\tilde{T}(A)$ denote the set $\{\tilde{T}(\mathbf{x}):~ \mathbf{x} \in A\}$. Note the important fact that $\tilde{T}(A)$ is again an element of $\mathcal{A}^*$. We have
\begin{eqnarray*}
\frac{1}{N} \sum_{n=1}^N \underbrace{\mathds{1}_{A}(T(\mathbf{x}_n))}_{=\mathds{1}_{\tilde{T}(A)}(\mathbf{x}_n)} - \underbrace{\mu (A)}_{=\prod_{s=1}^d G_s(a_s)} = \frac{1}{N} \sum_{n=1}^N \mathds{1}_{\tilde{T}(A)}(\mathbf{x}_n) - \lambda(\tilde{T}(A)).
\end{eqnarray*}
Taking absolute values and the supremum over all sets $A \in \mathcal{A}^*$, inequality~\eqref{thconcl} follows.\\

If all functions $G_1,\dots,G_d$ are invertible, then the mapping $A \mapsto T(A)$ is a bijection from the class $\mathcal{A}^*$ to itself, and $\tilde{T}$ is the inverse of $T$. Thus in this case the suprema $\sup_{A \in \mathcal{A}^*}$ and $\sup_{\tilde{T}(A):~A \in \mathcal{A}^*}$ coincide, which implies that in~\eqref{thconcl} there must be equality.
\end{proof}


\def\cprime{$'$}

\end{document}